\documentclass{gtpart}     
\usepackage{pinlabel}  
\usepackage{graphicx}  
\usepackage[all]{xy}
\usepackage{amscd}
\usepackage[labelsep=space]{caption}
\usepackage{float}
\usepackage{color}
\usepackage{marginnote}
\usepackage{comment}
\title{DECOMPOSITION OF LOOP SPACES AND PERIODIC PROBLEM ON $\pi_*$}
\author{Weidong Chen}
\givenname{Weidong}
\surname{Chen}
\address{}
\email{chenwd@nus.edu.sg}
\urladdr{}
\author{Jie Wu}
\givenname{Jie}
\surname{Wu}
\address{}
\email{matwuj@nus.edu.sg}
\urladdr{}

\arxivreference{none}
\arxivpassword{none}

\volumenumber{}
\issuenumber{}
\publicationyear{}
\papernumber{}
\startpage{}
\endpage{}
\doi{}
\MR{}
\Zbl{}
\received{}
\revised{}
\accepted{}
\published{}
\publishedonline{}
\proposed{}
\seconded{}
\corresponding{}
\editor{}
\version{}

\newtheorem{thm}{Theorem}[section]    
\newtheorem{prop}[thm]{Proposition} 
\newtheorem{coro}[thm]{Corollary}
\theoremstyle{definition}

\newtheorem*{ack}{Acknowledgements} 
\def\co{\colon\thinspace}

\DeclareMathOperator*{\hocolim}{hocolim}

\begin{document}

\begin{abstract}    
We provide a family of spaces localized at $2$,
whose stable homotopy groups are summands of their unstable homotopy groups.
Application to mod--$2$ Moore spaces are given.
\end{abstract}

\maketitle


\section{Introduction}

Homotopy theory is a central topic in the area of algebraic topology.
Understanding the relationship between the stable homotopy groups and unstable homotopy groups is an important question in homotopy theory.
Let $X$ be a $n$--connected pointed space.
Recall that the classical Freudenthal suspension theorem \cite{MR1556999} states that the canonical map
\begin{align*}
\pi_k(X) \rightarrow \pi_k^s(X)
\end{align*}
is an isomorphism if $k \leq 2n$ and an epimorphism if $k = 2n + 1$.
The Freudenthal suspension theorem relates the unstable homotopy groups to the stable homotopy groups.

Recently, a new interesting problem in this area has been proposed.
Namely, find spaces $X$ whose stable homotopy groups are summands of the unstable homotopy groups.
Beben and Wu \cite{mrth} gave examples of such spaces and applied their results to the Moore conjecture.
They showed that for a fixed odd prime $p$ and some $p$--localization of a CW--complex of finite type $X$,
there exists a sequence $\{ l_n \}$ which converges to infinity such that $\Omega \Sigma^{l_n} X$ is a homotopy retract of $\Omega X$.
Hence $\pi_{k+1}(\Sigma^{l_n} X)$ is a retract of $\pi_k(X)$.
Letting $\{ l_n \}$ converge to infinity,
the stable homotopy groups of $X$ are seen to be summands of its unstable homotopy groups.
Symbolically,
the group $\pi^s_*(X)$ is a summand of $\pi_*(X)$.
In this way,
Beben and Wu reduced the aforementioned problem to finding spaces $X$ together with a sequence  $\{ l_n \}$ that converges to infinity
such that $\Omega \Sigma^{l_n} X$ is a retract of $\Omega X$.
In this article,
we consider the case when $p=2$.
Our results are given as follows.
\begin{thm}\label{thm1}
For every $1 \leq i \leq n$, let $X_i$ be a path--connected 2--local CW--complex
such that $\bar{H}_*(X_i;\mathbb{Z}/2\mathbb{Z})$ is of dimension 2 with generators $u_i$, $v_i$ and $|u_i|<|v_i|$.
Then $\Omega\Sigma\wedge_{i=1}^{n} (\Sigma^{\frac{3^k-1}{2}(|u_i|+|v_i|)}{X_i})$ is a retract of $\Omega\Sigma\wedge_{i=1}^{n} {X_i}$
for every $k\geq 1$.
\end{thm}
As a consequence of Theorem \ref{thm1},
the stable homotopy groups of certain $2$--local spaces 
retract off
their regular homotopy groups.
\begin{coro}
\label{coro1}
For every $1 \leq i \leq n$, let $X_i$ be a path--connected 2--local CW--complex
such that $\bar{H}_*(X_i;\mathbb{Z}/2\mathbb{Z})$ is of dimension 2 with generators $u_i$, $v_i$ and $|u_i|<|v_i|$.
Let $b_k=\Sigma_{i=1}^{n}(\frac{3^k-1}{2}(|u_i|+|v_i|))$ and 
suppose that
$\wedge_{i=1}^{n} {X_i}$ is $(m-1)$--connected.
Then for large enough $k$ such that $j\leq b_k+2m$,
the group $\pi^s_j(\Sigma\wedge_{i=1}^{n} {X_i})$ is a homotopy retract of $\pi_{j+b_k}(\Sigma\wedge_{i=1}^{n} {X_i})$.
\end{coro}
Theorem \ref{thm1} deals with the case of finite wedge products of $2$--cell CW--complexes.
In the case of a single $2$--cell CW--complex,
this theorem can be strengthened by using a known decomposition of $\Omega\Sigma X$ \cite{MR2221079, MR1650426}.
The strengthened theorem is given as follows.
\begin{thm}\label{thm2}
	Let $X$ be a simply connected $2$--local CW--complex
	such that  $\bar{H}_*(X;\mathbb{Z}/2\mathbb{Z})$ is of dimension 2 with generators $u$, $v$ and $|u|<|v|$.
	Then we have
\begin{align*}
	\Omega\Sigma X \simeq \prod_j \Omega\Sigma^{1+k_j(|u|+|v|)} X \times \text{(some other space)},
\end{align*}
where $2<k_1<k_2<...$ are all prime numbers greater than 2.
\end{thm}
A fundamental problem in homotopy theory is to compute the homotopy groups of a given space.
We apply Theorem \ref{thm2} to compute $\mathbb{Z}/8\mathbb{Z}$--summands of the homotopy groups of mod--$2$ Moore spaces.
Let $\mathbb{R}\mathrm{P}^2$ be the projective plane. 
The $n$--dimensional mod--$2$ Moore space $P^n(2)$ is defined by $P^n(2)=\Sigma^{n-2}\mathbb{R}\mathrm{P}^2$ for $n\geq 2$.
$P^n(2)$ can be viewed as the homotopy cofibre of the degree 2 map $[2]\co S^{n-1}\rightarrow S^{n-1}$.
That is to say,
the cell complex $P^n(2)$ is obtained by attaching an $n$--cell to $S^{n-1}$ and the attaching map is give by the degree $2$ map.


This problem was studied earlier by  Cohen and Wu \cite{MR1320988}.
They noted that a $\mathbb{Z}/8\mathbb{Z}$--summands of $\pi_*(P^{4n+1}(2))$ can be found in $\pi_{120n-14}(P^{4n+1}(2))$.
They also asked whether $\pi_*(P^{4n+1}(2))$ has $\mathbb{Z}/8\mathbb{Z}$--summands in lower degrees.
The following corollary of Theorem \ref{thm2} answers Cohen and Wu's question in the affirmative.

\begin{coro}\label{coro2}
There exists homotopy equivalences for every $n\geq 1$.
\begin{enumerate}
	\item $\Omega P^{4n}(2)\simeq \Omega P^{(8k+4)n-3k}(2)\times \text{(some other space)}\\$
	Thus $\pi_{(16k+8)n-6k-2}(P^{4n}(2))$ contains a $\mathbb{Z}/8\mathbb{Z}$--summand, for all $k\in\mathbb{Z}^{\geq 0}$ such that $k\equiv 2 \pmod 4$.
	\item $\Omega P^{4n+1}(2)\simeq \Omega P^{(8k+4)n+1-k}(2)\times \text{(some other space)}\\$
	Thus $\pi_{(16k+8)n-2k}(P^{4n+1}(2))$ contains a $\mathbb{Z}/8\mathbb{Z}$--summand, for all $k\in\mathbb{Z}^{\geq 0}$ such that $k\equiv 3 \pmod 4$.
	\item $\Omega P^{4n+2}(2)\simeq \Omega P^{(8k+4)n+2+k}(2)\times \text{(some other space)}\\$
	Thus $\pi_{(16k+8)n+2k+2}(P^{4n+3}(2))$ contains a $\mathbb{Z}/8\mathbb{Z}$--summand, for all $k\in\mathbb{Z}^{\geq 0}$ such that $k\equiv 0 \pmod 4$.
	\item $\Omega P^{4n+3}(2)\simeq \Omega P^{(8k+4)n+3+3k}(2)\times \text{(some other space)}\\$
	Thus $\pi_{(16k+8)n+6k+4}(P^{4n+3}(2))$ contains a $\mathbb{Z}/8\mathbb{Z}$--summand, for all $k\in\mathbb{Z}^{\geq 0}$ such that $k\equiv 1 \pmod 4$.
\end{enumerate}
\end{coro}
In particular,
there is a $\mathbb{Z}/8\mathbb{Z}$--summand in $\pi_{56n-6}(P^{4n+1}(2))$.
This is of a degree lower than that given in \cite{MR1320988}.

This article is organized as follows.
In Section 2,
we introduce some notations and basic properties.
The proofs of Theorem \ref{thm1} and \ref{thm2} are given in Section 3.
The proofs of Corollary of \ref{coro1} and \ref{coro2} and some remarks are given in Section 4.
\begin{ack}
The 2nd author gratefully acknowledges the assistance of Singapore Ministry of Education research grants AcRF Tier 2 (WBS No.R-146-000-143-112)
and a grant (No. 11028104) of NSFC of China.
We wish to thank the referee most warmly for numerous suggestions that have improved the exposition of this paper.
The authors are greatful to Weiyu Tan for proofreading this paper.
\end{ack}
\section{Preliminary}
Let $X$ be a path--connected, $2$--local CW--complex of finite type
and
let $X^{(n)}$ be the $n$--fold self smash product of the space $X$.
Let $S_k$ denote the symmetric group on $k$ letters and let $\mathbb{Z}_{(2)}(S_n)$ denote the group ring
over the $2$ local integer $\mathbb{Z}_{(2)}$ generated by $S_k$.

Consider the action of $\mathbb{Z}_{(2)}(S_n)$ on $\Sigma X^{(n)}$ by permuting coordinates and taking the summations.
For any $\delta\in\mathbb{Z}_{(2)}(S_n)$ we obtain a map
\begin{align*}
\tilde\delta \co \Sigma X^{(n)} \rightarrow \Sigma X^{(n)}.
\end{align*}
Let $V=\bar H_* (X;\mathbb{Z}/2\mathbb{Z})$,
which is a graded $\mathbb{Z}/2\mathbb{Z}$--module.
We use 
\begin{align*}
\Sigma V^{\otimes n}= \bar H^*( S^1 ; \mathbb{Z}/2\mathbb{Z} )  \otimes V^{\otimes n}
\end{align*}
to denote the $\mathbb{Z}/2\mathbb{Z}$ reduced homology of $\Sigma X^{(n)}$.
Therefore $\tilde\delta$ induces a map
\begin{align*}
\tilde\delta_* \co \Sigma V^{\otimes n} \longrightarrow \Sigma V^{\otimes n}
\end{align*}
by permuting factors.

Define the \textit{Dynkin--Specht--Wever} elements inductively.
Start with $\beta_2=1-(1,2)\in \mathbb{Z}_{(2)}(S_2)$.
Then 
\begin{align*}
	\beta_n = \beta_{n-1}\wedge id - (1,2,...,n) \circ (\beta_{n-1}\wedge id).
\end{align*}
The element $\beta_n$ induces a map
\begin{align*}
\tilde\beta_n \co \Sigma X^{(n)} \rightarrow \Sigma X^{(n)}.
\end{align*}
Let $\iota_1$ denote the generator of the mod--$2$ reduced homology of $S^1$,
then
\begin{align*}
	\tilde\beta_{n*}(\iota_1 \otimes x_1 \otimes \ldots \otimes x_n)=\iota_1 \otimes [[\ldots[[x_1,x_2],\ldots x_{n-1}],x_n],
\end{align*}
where $[[\ldots[[x_1,x_2],\ldots x_{n-1}],x_n]\in (\mathbb{Z}/2\mathbb{Z})^{\otimes n}$ denotes the commutators.
$\tilde\beta_{n*}\circ\tilde\beta_{n*}=n\tilde\beta_{n*}$ following \cite{MR1320988}.
Hence if $n$ is an odd integer,
then $ \frac{1}{n}\tilde\beta_{n*} \circ \frac{1}{n}\tilde\beta_{n*} = \frac{1}{n}\tilde\beta_{n*}$.

Denote by $\hocolim_{f} \Sigma X^{(n)}$ the mapping telescope of the following a sequence of maps:
\begin{align*}
	\Sigma X^{(n)} \stackrel{f}{\longrightarrow} \Sigma X^{(n)} \stackrel{f}{\longrightarrow} \cdots
\end{align*}
For an odd integer $n$,
the elements
$\frac{1}{n}\tilde\beta_{n*}$ and $(id_{\Sigma X^{(n)}}-\frac{1}{n}\tilde\beta_{n})_*$ are orthogonal idempotents.
Since $\frac{1}{n}\tilde\beta_{n*} \circ (id_{\Sigma X^{(n)}}-\frac{1}{n}\tilde\beta_{n})_*$ and
$(id_{\Sigma X^{(n)}}-\frac{1}{n}\tilde\beta_{n})_* \circ \frac{1}{n}\tilde\beta_{n*}$ are trivial,
the following the composite is a homotopy equivalence:
\begin{align*}
\Sigma X^{(n)} \xrightarrow{comult} \Sigma X^{(n)} \vee \Sigma X^{(n)} 
\rightarrow \hocolim_{\frac{1}{n} \tilde\beta_n} \Sigma X^{(n)} 
\vee \hocolim_{id_{\Sigma X^{(n)}}-\frac{1}{n}\tilde\beta_{n}} \Sigma X^{(n)}.
\end{align*}
This is because the induced map on the homology with coefficients in $2$--local integers is an isomorphism.
Hence $\hocolim_{\frac{1}{n} \tilde\beta_n} \Sigma X^{(n)}$ retracts off $\Sigma X^{(n)}$.
Let $\tilde L_n(X) = \hocolim_{\frac{1}{n} \tilde\beta_n} \Sigma X^{(n)}$.
Let $p \co \Sigma X^{(n)} \rightarrow \tilde L_n(X)$ be the projection
and let $i \co \tilde L_n(X) \rightarrow \Sigma X^{(n)}$ be the canonical inclusion,
then $\frac{1}{n}\tilde\beta_{n*}$ is identical to the composition:
\begin{align}
	H_*(\Sigma X^{(n)};\mathbb{Z}/2\mathbb{Z}) \stackrel{p_*}{\rightarrow} H_*(\tilde L_n(X);\mathbb{Z}/2\mathbb{Z}) \stackrel{i_*}{\rightarrow} H_*(\Sigma X^{(n)};\mathbb{Z}/2\mathbb{Z}). \label{Eq:pi}
\end{align}
In the special case when $X$ is a suspension,
as in \cite{MR1650426} we let $L_n(X)=\hocolim_{\frac{1}{n} \beta_n} X^{(n)}$.
In this case $\tilde L_n(X) \simeq \Sigma L_n(X)$.

It is well known that the mod--$2$ reduced homology of $\mathbb{R}\mathrm{P}^2\wedge \mathbb{R}\mathrm{P}^2$ contains a spherical class of degree $3$.
We generalize this fact to any path--connected $2$--cell CW--complex.
\begin{prop}
\label{prop1}
Let $X$ be a $2$--local $2$--cell path--connected CW--complex
such that $\bar{H}_*(X;\mathbb{Z}/2\mathbb{Z})$ is of dimension 2 with generators $u$ and $v$ such that $|v|=m$ and $|u|=n$.
Then there exists a map
\begin{align*}
		\alpha \co S^{m+n}\rightarrow X\wedge X
\end{align*}
such that the image in $\alpha_*$ of induced map on mod--$2$ homology
is generated by $[u,v]\in H_*(X\wedge X;\mathbb{Z}/2\mathbb{Z})$.
\end{prop}
\begin{proof}
When $m=n$, $X$ is just a wedge product of two spheres
and the statement is trivial.
Assume without loss of generality that $m>n$.
Let $f \co S^{m-1}\rightarrow S^n$ be the attaching map of the $m$--cell to the $n$--cell of $X$.
Consider the homotopy cofibration
\begin{align*}
	X\wedge S^{m-1} \stackrel{id_X \wedge f}{\longrightarrow} X\wedge S^{n} \rightarrow X\wedge X.
\end{align*}
Take the
$(2m-1)$--skeleton
of $X\wedge X$ and note that $S^n\wedge S^{m-1}\simeq sk_{2m-2}(X\wedge S^{m-1})$,
hence we obtain a homotopy cofibration
\begin{align*}
	S^n\wedge S^{m-1} \stackrel{i \wedge f}{\longrightarrow} X\wedge S^{n} \rightarrow sk_{2m-1}(X\wedge X).
\end{align*}
Next we will show that $i \wedge f$ is null--homotopic.
Since $i \wedge f$ is homotopic to the composite
\begin{align*}
	S^n \wedge S^{m-1} \stackrel{id_{S^n}\wedge f}{\longrightarrow} S^n\wedge S^n
	\stackrel{i\wedge id_{S^n}}{\longrightarrow} X\wedge S^n,
\end{align*}
we obtain the homotopy commutative diagram
\begin{displaymath}
\begin{xy}
(0,20)*+{S^n\wedge S^{m-1}}="01";(40,20)*+{S^n \wedge S^n}="11";
(0,0)*+{S^{m-1}\wedge S^n}="00";(40,0)*+{S^n \wedge S^n}="10";(80,0)*+{X\wedge S^n}="20";
{\ar^{id_{S^n}\wedge f} "01";"11"};
{\ar^{\tau'} "01";"00"};{\ar^{\tau} "11";"10"};
{\ar^{f\wedge id_{S^n}} "00";"10"};{\ar^{i\wedge id_{S^n}} "10";"20"};
\end{xy}
\end{displaymath}
where the bottom row is homotopy cofibration and $\tau'$ and $\tau$ are switching maps.

There two cases:
the map $S^n\wedge S^n \stackrel{\tau}{\rightarrow} S^n\wedge S^n$ has degree either 1 or -1.
If $\deg(\tau)=1$, then $\tau\simeq id_{S^{2n}}$ and
\begin{align*}
	i\wedge f  	&\simeq (i\wedge id_{S^n}) \circ (id_{S^n}\wedge f) \\
				&\simeq (i\wedge id_{S^n}) \circ \tau \circ (id_{S^n}\wedge f) \\
				&\simeq (i\wedge id_{S^n}) \circ (f\wedge id_{S^n}) \circ \tau' \\
				&\simeq *.
\end{align*}
If $\deg(\tau)=-1$, since $id_{S^n}\wedge f$ is a suspension, we have
\begin{align*}
	(id_{S^n}\wedge f) \circ [-1] \simeq \tau \circ (id_{S^n}\wedge f).
\end{align*}
Thus we obtain that
\begin{align*}
	id_{S^n}\wedge f 	&\simeq (id_{S^n}\wedge f) \circ [-1] \circ [-1] \\
									 	&\simeq \tau \circ (id_{S^n}\wedge f) \circ [-1] \\
									 	&\simeq (f\wedge id_{S^n}) \circ \tau' \circ [-1].
\end{align*}
It follows that
\begin{align*}
	i\wedge f	&\simeq (i\wedge id_{S^n}) \circ (id_{S^n}\wedge f) \\
						&\simeq (i\wedge id_{S^n}) \circ (f\wedge id_{S^n}) \circ \tau' \circ [-1] \\
						&\simeq *.
\end{align*}
Thus in either case,
$i\wedge f$ is null--homotopic. 
Therefore
\begin{align*}
	sk_{2m-1}(X\wedge X) \simeq
	(X\wedge S^{n}) \vee S^{m+n}.
\end{align*}
Hence we can define $\alpha$ by the following composite of canonical inclusions:
\begin{align*}
	S^{m+n}\hookrightarrow sk_{2m-1}(X\wedge X) \hookrightarrow X\wedge X.
\end{align*}
Because $[u,v]$ is the only primitive generator of $H_{m+n}(X\wedge X;\mathbb{Z}/2\mathbb{Z})$,
one gets $\alpha_*(\iota_{m+n})=[u,v]\in H_*(X\wedge X;\mathbb{Z}/2\mathbb{Z})$.
\end{proof}
\section{Decomposition of Loop Spaces and Proofs of Theorem \ref{thm1} and \ref{thm2}}
Recall that for an odd integer $n$, the space $\tilde L_n(X)$ is a homotopy retract of $\Sigma X^{(n)}$.
By studying the $2$--local decomposition of $\Sigma X^{(n)}$,
one can investigate the spaces $\tilde L_n(X)$.
In \cite{MR2002820}, the finest $2$--primary splitting of $X^{(n)}$ is given.

Suppose that $\bar{H}_*(X;\mathbb{Z}/2\mathbb{Z})$ is generated by two elements $u$ and $v$ with $|u|<|v|$.
Proposition \ref{prop1} implies that,
if $n=2k+1$ for some non-negative integer $k$,
then there exists a canonical inclusion $\Sigma^{1+k(|u|+|v|)}X \hookrightarrow \Sigma X^{(n)}$.
\begin{prop}
\label{prop2}
Let $X$ be a path--connected $2$--local $2$--cell CW--complex such that $\bar{H}_*(X;\mathbb{Z}/2\mathbb{Z})$ is of dimension 2 with generators $u$ and $v$ such that $|u|<|v|$.
For every odd integer $n \geq 3$,
let $p_n$ be the natural projection defined in Equation \eqref{Eq:pi},
then the following composite has a left homotopy inverse:
\begin{align*}
	\Sigma^{1+k(|u|+|v|)} X \rightarrow \Sigma X^{(n)} \xrightarrow{p_n} \tilde L_{n}(X).
\end{align*}
\end{prop}
\begin{proof}
For every odd integer $n \geq 3$,
let $i_n$ be the inclusion defined as in Equation \eqref{Eq:pi}.
Recall that 
\begin{align*}
(i_n\circ p_n)_*=\tilde\beta_n \co H_*(\Sigma X^{(n)}\mathbb{Z}/2\mathbb{Z}) \rightarrow H_*(\Sigma X^{(n)}\mathbb{Z}/2\mathbb{Z}).
\end{align*}
Let $\alpha$ be the spherical class given in Proposition \ref{prop1}
and let $\phi_1$ be the composite:
\begin{align*}
	\Sigma^{1+|u|+|v|} X \simeq \Sigma X \wedge S^{|u|+|v|}	\stackrel{id \wedge \alpha}{\longrightarrow} \Sigma X\wedge X\wedge X
														\stackrel{i_3\circ p_3}{\longrightarrow} \Sigma X\wedge X\wedge X.
\end{align*}
Take the mod--$2$ reduced homology:
\begin{align*}
	{\phi_1}_*(\iota_1\otimes u\otimes\iota_{|u|+|v|})&=(i_{3*}\circ p_{3*}) (\iota_1\otimes(uuv+uvu)) \\
	&=  \tilde\beta_{3*}(\iota_1\otimes(uuv+uvu))=\iota_1\otimes[[u,v],u].
\end{align*}
Similarly ${\phi_1}_*(\iota_1\otimes v\otimes\iota_{|u|+|v|})=\iota_1\otimes[[u,v],v]$.
Since $\bar H_*(\tilde L_3(X))$ is of dimension $2$,
the composite $p_3\circ\phi_1$ induces an isomorphism in mod--$2$ homology.
Because the spaces involved are CW--complexes of finite type,
hence we have the homotopy equivalence:
\begin{align*}
	p_3\circ\phi_1 \co \Sigma^{1+|u|+|v|} X \rightarrow \tilde L_3(X).
\end{align*}
Define $\phi_{k} \co \Sigma^{1+k(|u|+|v|)} X \rightarrow \Sigma X^{(2k+1)}$ inductively as the composite:
\begin{align*}
	\Sigma^{1+k(|u|+|v|)} X \simeq \Sigma^{1+(k-1)(|u|+|v|)} X \wedge S^{|u|+|v|}  \stackrel{\phi_{k-1}\wedge\alpha}{\longrightarrow} \Sigma X^{(2k-1)} \wedge X^{(2)}
													\stackrel{i_{2k+1} \circ p_{2k+1} }{\longrightarrow} \Sigma X^{(2k+1)}.
\end{align*}
Set $\varphi_1=p_3$.
Define $\varphi_{k} \co \Sigma X^{(2k+1)} \rightarrow \Sigma^{1+k(|u|+|v|)} X$ inductively as the composite:
\begin{align*}
	\Sigma X^{(2k+1)} 
	\stackrel{\varphi_{k-1} \wedge id}{\longrightarrow} \Sigma^{1+(k-1)(|u|+|v|)} X \wedge X \wedge X
	\stackrel{\Sigma^{(k-1)(|u|+|v|)} p_3}{\longrightarrow} \Sigma^{1+k(|u|+|v|)} X.
\end{align*}
Since $i_{2k+1} \circ p_{2k+1}$ factors through $\tilde L_{2k+1}(X)$,
the composite $\varphi_{k}\circ\phi_{k}$ factors through $\tilde L_{2k+1}(X)$.
Let $ad(x)(y)=[y,x]$ and $ad^{i+1}(x)(y)=[ad^i(x)(y),x]$ for $i\geq 1$.
The coefficients are taken mod--$2$,
hence $ad([u,v])(u)=[[u,v],u]=[u,[u,v]]$.

Thus it is sufficient to show that the following composite is a homotopy equivalence for all $k\geq 1$:
\begin{align*}
	\Sigma^{1+k(|u|+|v|)} X \stackrel{\phi_k}{\rightarrow} \Sigma X^{(2k+1)} \stackrel{\varphi_k}{\rightarrow} \Sigma^{1+k(|u|+|v|)} X.
\end{align*}
Since the spaces involved are CW--complexes of finite type,
it is enough to show that $\varphi_{k}\circ\phi_{k}$ induces an isomorphism on mod--$2$ homology.
Explicitly,
we will prove the following statements by induction on $k$: 
\begin{equation}\label{Eq:2}
\begin{split}
	\varphi_{k*} \circ \phi_{k*} (\iota_{1+k(|u|+|v|)} \otimes u )=& \varphi_{k*} (\iota_1 \otimes ad^{k}([u,v])(u))=\iota_{1+k(|u|+|v|)} \otimes u, \\
	\varphi_{k*} \circ \phi_{k*} (\iota_{1+k(|u|+|v|)} \otimes v )=& \varphi_{k*} (\iota_1 \otimes ad^{k}([u,v])(v))=\iota_{1+k(|u|+|v|)} \otimes v. 
\end{split}
\end{equation}
The case of $k=1$ has already been shown.
Suppose that the statement is true for all $k'<k$.
We have
\begin{align*}
	\tilde\beta_{3*}(\iota_1\otimes ad([u,v])(u))=\tilde\beta_{3*}(\iota_1\otimes[[u,v],u])=\iota_1\otimes[[u,v],u]=\iota_1\otimes ad([u,v])(u).
\end{align*}
Also
\begin{align*}
	&\tilde\beta_{2k+1*}(\iota_1\otimes[u,v]\otimes(ad^{k-1}([u,v])(u)) \\
	=&\tilde\beta_{2k+1*}(\iota_1\otimes uv \otimes(ad^{k-1}([u,v])(u)) + 
	\tilde\beta_{2k+1*}
	(\iota_1\otimes vu\otimes(ad^{k}([u,v])(u))\\
	=&2\tilde\beta_{2k+1*}(\iota_1\otimes uv \otimes(ad^{k-1}([u,v])(u)) \\
	=&0.
\end{align*}
Hence we obtain
\begin{align*}
	&\tilde\beta_{2k+1*}(\iota_1\otimes ad^{k}([u,v])(u))\\
	=&\tilde\beta_{2k+1*}(\iota_1\otimes[ad^{k-1}([u,v])(u),[u,v]]) \\
	=&\tilde\beta_{2k+1*}(\iota_1\otimes(ad^{k-1}([u,v])(u))\otimes[u,v])+\tilde\beta_{2k+1*}(\iota_1\otimes[u,v]\otimes(ad^{k-1}([u,v])(u)))\\
	=&\tilde\beta_{2k+1*}(\iota_1\otimes(ad^{k-1}([u,v])(u))\otimes[u,v]).
\end{align*}
Furthermore
\begin{align*}
	&\tilde\beta_{2k+1*}(\iota_1\otimes(ad^{k-1}([u,v])(u))\otimes[u,v])\\
	=&\tilde\beta_{2k+1*}(\iota_1\otimes(ad^{k-1}([u,v])(u))\otimes uv) + \tilde\beta_{2k+1*}(\iota_1\otimes(ad^{k-1}([u,v])(u))\otimes vu)\\
	=&[[[\tilde\beta_{2k-1*}(\iota_1\otimes(ad^{k-1}([u,v])(u)))],u],v] + [[[\tilde\beta_{2k-1*}(\iota_1\otimes(ad^{k-1}([u,v])(u)))],v],u]\\
	=&[\tilde\beta_{2k-1*}(\iota_1\otimes ad^{k-1}([u,v])(u)),[u,v]] \\
	=&ad([u,v])(\tilde\beta_{2k-1*}(\iota_1\otimes ad^{k-1}([u,v])(u))).
\end{align*}
where the third equality is due to Jacobi identity.
Thus we have
\begin{align*}
	&\tilde\beta_{2k+1*}(\iota_1\otimes ad^{k}([u,v])(u))\\
	=&ad([u,v])(\tilde\beta_{2k-1*}(\iota_1\otimes ad^{k-1}([u,v])(u)))\\
	=&ad^2([u,v])(\tilde\beta_{2k-3*}(\iota_1\otimes ad^{k-2}([u,v])(u))).
\end{align*}
Hence we obtain that
\begin{align*}
	&\tilde\beta_{2k+1*}(\iota_1\otimes ad^{k}([u,v])(u))
	=\iota_1\otimes ad^{k}([u,v])(u).
\end{align*}
By induction hypothesis, we have
\begin{align*}
	\phi_{k-1*}(\iota_{1+(k-1)(|u|+|v|)}\otimes u) = \iota_1\otimes ad^{k-1}([u,v])(u).
\end{align*}
It follows that
\begin{align*}
	\phi_{k*}(\iota_{1+k(|u|+|v|)}\otimes u)
	&=(i_{2k+1}\circ p_{2k+1})_* (\phi_{k-1*} (\iota_{1+(k-1)(|u|+|v|)}\otimes u) \otimes [u,v])\\
	&=\tilde\beta_{2k+1*}(\iota_1\otimes ad^{k-1}([u,v])(u) \otimes [u,v])\\
	&=\tilde\beta_{2k+1*}(\iota_1\otimes ad^{k}([u,v])(u))\\
	&=\iota_1\otimes ad^{k}([u,v])(u).
\end{align*}
Similarly
\begin{align*}
	\phi_{k*}(\iota_{1+k(|u|+|v|)}\otimes v )
	&=\iota_1\otimes ad^{k}([u,v])(v).
\end{align*}
Also we have
\begin{align*}
	&(\varphi_{k*}\circ\tilde\beta_{2k+1*})(\iota_1\otimes ad^{k}([u,v])(u)) \\
	=&\varphi_{k*}(\iota_1\otimes ad^{k-1}([u,v])(u)\otimes [u,v]))\\
	=&(\Sigma^{(k-1)(|u|+|v|)}p_3)_*\circ((\varphi_{k-1*}\otimes id)(\iota_1\otimes ad^{k-1}([u,v])(u)\otimes [u,v]))\\
	=&(\Sigma^{(k-1)(|u|+|v|)}p_3)_*\circ(\varphi_{k-1*}((\iota_1\otimes ad^{k-1}([u,v])(u))\otimes [u,v])\\
	=&(\Sigma^{(k-1)(|u|+|v|)}p_3)_*(\iota_{1+(k-1)(|u|+|v|)}\otimes u\otimes [u,v])\\
	=&\iota_{1+k(|u|+|v|)}\otimes u.
\end{align*}
Since
$\tilde\beta_{2k+1*}=i_{2k+1*}\circ p_{2k+1*}$ and
$\tilde\beta_{2k+1*}=\tilde\beta_{2k+1*}\circ\tilde\beta_{2k+1*}$,
we have
\begin{align*}
\varphi_{k*}\circ\phi_{k*}=\varphi_{k*}\circ\tilde\beta_{2k+1*}\circ\phi_{k*}.
\end{align*}
Hence
\begin{align*}
	&(\varphi_{k*}\circ\phi_{k*})(\iota_{1+k(|u|+|v|)}\otimes u)\\
	=&(\varphi_{k*}\circ\tilde\beta_{2k+1*}\circ\phi_{k*})(\iota_{1+(k-1)(|u|+|v|)}\otimes u\otimes \iota_{|u|+|v|})\\
	=&\iota_{1+k(|u|+|v|)}\otimes u.
\end{align*}
Similarly we obtain
\begin{align*}
	&(\varphi_{k*}\circ\phi_{k*})(\iota_{1+k(|u|+|v|)}\otimes v)\\
	=&(\varphi_{k*}\circ\tilde\beta_{2k+1*}\circ\phi_{k*})(\iota_{1+(k-1)(|u|+|v|)}\otimes v\otimes \iota_{|u|+|v|})\\
	=&\iota_{1+k(|u|+|v|)}\otimes v.
\end{align*}
This completes the induction step.
Thus \eqref{Eq:2} holds.
As noted earlier,
this implies the composite $\varphi_{k}\circ\phi_{k}$ is a homotopy equivalence,
which completes the proof.
\end{proof}
We can obtain a weaker result when the space studied is a finite wedge product of $2$--cell CW--complexes.
\begin{prop}
\label{prop3}
For every $1 \leq k \leq n$, let $X_k$ be a path--connected $2$--local $2$--cell CW--complexes
such that $\bar{H}_*(X_k;\mathbb{Z}/2\mathbb{Z})$ is of dimension 2 with generators $u_k$ and $v_k$ such that $|u_k|<|v_k|$.
Then the following map has a left homotopy inverse:
\begin{align*}
	\Sigma\wedge_{k=1}^{n} \Sigma^{|u_k|+|v_k|}{X_k} \rightarrow \tilde L_{3}(\wedge_{k=1}^{n} {X_k}).
\end{align*}
\end{prop}
\begin{proof}
As shown in the proof of Proposition \ref{prop2},
for each space $X_k$ we have a canonical projection 
\begin{align*}
p^k_3 \co \Sigma X_k^{(3)} \rightarrow \Sigma^{1+|u_k|+|v_k|}X_k \simeq \tilde L_3 (X_k)
\end{align*}
and a canonical inclusion
\begin{align*}
i^k_3 \co \Sigma^{1+|u_k|+|v_k|}X_k \simeq \tilde L_3 (X_k) \rightarrow \Sigma X_k^{(3)}.
\end{align*}
The map $p^1_3,\ldots,p^n_3$ induce the projection
\begin{align*}
	\bar p \co \Sigma\wedge_{k=1}^{n} {X_k}^{(3)}\rightarrow \Sigma\wedge_{k=1}^{n} \Sigma^{|u_k|+|v_k|}{X_k}.
\end{align*}
The map $i^1_3,\ldots,i^n_3$ induce the inclusion
\begin{align*}
	\bar i \co \Sigma\wedge_{k=1}^{n} \Sigma^{|u_k|+|v_k|}{X_k} \rightarrow \Sigma\wedge_{k=1}^{n} {X_k}^{(3)}.
\end{align*}
Let $\theta$ be the composite
\begin{align*}
	\Sigma\wedge_{k=1}^{n} {X_k}^{(3)} \stackrel{\Sigma\tau}{\longrightarrow} \Sigma(\wedge_{k=1}^{n} {X_k})^{(3)} 
																				\stackrel{p}{\longrightarrow} \tilde L_{3}(\wedge_{k=1}^{n} {X_k})
																				\stackrel{i}{\longrightarrow} \Sigma(\wedge_{k=1}^{n} {X_k})^{(3)}
																				\stackrel{\Sigma\tau'}{\longrightarrow}\Sigma\wedge_{k=1}^{n} {X_k}^{(3)},
\end{align*}
where the maps $\tau$ and $\tau'$ switch positions.
Explicitly,
for $x_k,y_k,z_k\in H_*(X_k;\mathbb{Z}/2\mathbb{Z})$,
the maps $\tau$ and $\tau'$ induce the following maps on homology:
\begin{align*}
	&\tau_* (\otimes_{k=1}^n x_k y_k z_k) = (\Pi_{k=1}^n x_k)\otimes(\Pi_{k=1}^n y_k)\otimes(\Pi_{k=1}^n z_k),\\
	&\tau'_* ((\Pi_{k=1}^n x_k)\otimes(\Pi_{k=1}^n y_k))\otimes(\Pi_{k=1}^n z_k)) = \otimes_{k=1}^n x_k y_k z_k.
\end{align*}
Recall that for the canonical inclusion $i$ and canonical projection $p$,
we have $\tilde\beta_{3*}= i_* \circ p_*$.
Then $\theta_*(\iota_1\otimes_{k=1}^n x_k y_k z_k)$ is given by the following.
\begin{align*}
	\iota_1\otimes_{k=1}^n x_k y_k z_k 	&\stackrel{(\Sigma\tau)_*}{\longmapsto} 			\iota_1\otimes(\Pi_{k=1}^n x_k)\otimes(\Pi_{k=1}^n y_k)\otimes(\Pi_{k=1}^n z_k)\\
																			&\stackrel{(\tilde\beta_3)_*}{\longmapsto} 	\iota_1\otimes[[(\Pi_{k=1}^n x_k),(\Pi_{k=1}^n y_k)],(\Pi_{k=1}^n z_k)] \\
																			&\stackrel{(\Sigma\tau)_*}{\longmapsto}				\iota_1\otimes(\otimes_{k=1}^n x_k y_k z_k + \otimes_{k=1}^n y_k x_k z_k + \otimes_{k=1}^n z_k x_k y_k + \otimes_{k=1}^n z_k y_k x_k).
\end{align*}
Let $\gamma_1$, $\gamma_2$, $\gamma_3$, $\gamma_4$ be the mapping defined by
\begin{align*}
\gamma_1(\iota_1\otimes_{k=1}^n x_k y_k z_k) &= \iota_1\otimes(\otimes_{k=1}^n x_k y_k z_k),\\
\gamma_2(\iota_1\otimes_{k=1}^n x_k y_k z_k) &= \iota_1\otimes(\otimes_{k=1}^n y_k x_k z_k),\\
\gamma_3(\iota_1\otimes_{k=1}^n x_k y_k z_k) &= \iota_1\otimes(\otimes_{k=1}^n z_k x_k y_k),\\
\gamma_4(\iota_1\otimes_{k=1}^n x_k y_k z_k) &= \iota_1\otimes(\otimes_{k=1}^n z_k y_k x_k).
\end{align*}
Therefore $\theta_* = \gamma_1 + \gamma_2 + \gamma_3 + \gamma_4$
and $(\bar p\circ\theta)_* = \bar p_* \circ \theta_* = \bar p_* \circ \gamma_1 + \bar p_* \circ \gamma_2 + \bar p_* \circ \gamma_3 + \bar p_* \circ \gamma_4$.

Since $0=\tilde\beta_{3*}(\iota_1\otimes u_k u_k v_k)= i^k_{3*} \circ p^k_{3*}(\iota_1\otimes u_k u_k v_k)$ and $i^k_{3*}$ is a monomorphism,
one gets $p^k_{3*}(\iota_1\otimes u_k u_k v_k)=0$.
Recall that $\bar p$ is induced by $p^1_3,\ldots,p^n_3$.
If $x_l y_l z_l = u_l u_l v_l$ for some $1\leq l \leq n$,
then $\bar p_*(\iota_1\otimes_{k=1}^n x_k y_k z_k )=0$.
Since tensor product is bilinear,
\begin{align*}
\otimes_{k=1}^n(u_k u_k v_k + v_k u_k u_k) = \Sigma_{j=1}^{2^n} \otimes_{k=1}^n x_k^j y_k^j z_k^j
\end{align*}
where $x_k^j y_k^j z_k^j = u_k u_k v_k$ or $v_k u_k u_k$ for $1\leq j \leq 2^n$.

Hence if $x_l^j y_l^j z_l^j = u_l u_l v_l$ for some $1\leq l \leq n$, we have
\begin{align*}
\bar p_* \circ \gamma_1 (\iota_1\otimes_{k=1}^n x_k^j y_k^j z_k^j) = \bar p_* (\iota_1\otimes_{k=1}^n x_k^j y_k^j z_k^j ) 
= \bar p_*(\iota_1\otimes ... \otimes u_l u_l v_l \otimes ...) = 0.
\end{align*}
It follows that $\bar p_* \circ \gamma_1 (\iota_1\otimes_{k=1}^n x_k^j y_k^j z_k^j)$ is non--zero if and only if
\begin{align*}
\otimes_{k=1}^n x_k^j y_k^j z_k^j = \otimes_{k=1}^n v_k u_k u_k.
\end{align*}
We must have
\begin{align*}
\bar p_* \circ \gamma_1 (\iota_1\otimes_{k=1}^n(u_k u_k v_k + v_k u_k u_k))
=&\bar p_* \circ \gamma_1 ( \Sigma_{j=1}^{2^n} \iota_1\otimes_{k=1}^n x_k^j y_k^j z_k^j )\\
=&\Sigma_{j=1}^{2^n} \bar p_* \circ \gamma_1 ( \iota_1\otimes_{k=1}^n x_k^j y_k^j z_k^j )\\
=&\bar p_* \circ \gamma_1 (\iota_1\otimes_{k=1}^n v_k u_k u_k)\\
=&\bar p_*(\iota_1\otimes_{k=1}^n v_k u_k u_k)\\
=&\iota_1\otimes_{k=1}^n(\iota_{|u_k|+|v_k|}\otimes u_k).
\end{align*}
Similarly
\begin{align*}
\bar p_* \circ \gamma_2 (\iota_1\otimes_{k=1}^n(u_k u_k v_k + v_k u_k u_k))
=& \bar p_* \circ \gamma_2 ( \Sigma_{j=1}^{2^n} \iota_1\otimes_{k=1}^n x_k^j y_k^j z_k^j ) \\
=& \Sigma_{j=1}^{2^n} \bar p_* \circ \gamma_2 ( \iota_1\otimes_{k=1}^n x_k^j y_k^j z_k^j ) \\
=& \bar p_* \circ \gamma_2 (\iota_1\otimes_{k=1}^n v_k u_k u_k) \\
=& \bar p_*(\iota_1\otimes_{k=1}^n u_k v_k u_k) \\
=& \iota_1\otimes_{k=1}^n(\iota_{|u_k|+|v_k|}\otimes u_k),
\end{align*}
\begin{align*}
\bar p_* \circ \gamma_4 (\iota_1\otimes_{k=1}^n(u_k u_k v_k + v_k u_k u_k))
=& \bar p_* \circ \gamma_4 ( \Sigma_{j=1}^{2^n} \iota_1\otimes_{k=1}^n x_k^j y_k^j z_k^j ) \\
=& \Sigma_{j=1}^{2^n} \bar p_* \circ \gamma_4 ( \iota_1\otimes_{k=1}^n x_k^j y_k^j z_k^j ) \\
=& \bar p_* \circ \gamma_4 (\iota_1\otimes_{k=1}^n u_k u_k v_k) \\
=& \bar p_*(\iota_1\otimes_{k=1}^n v_k u_k u_k) \\
=& \iota_1\otimes_{k=1}^n(\iota_{|u_k|+|v_k|}\otimes u_k).
\end{align*}
Also notice that
$\bar p_* \circ \gamma_3 ( \otimes_{k=1}^n x_k^j y_k^j z_k^j )$
is non--zero for $1\leq j \leq 2^n$.
We get
\begin{align*}
\bar p_* \circ \gamma_3 (\iota_1\otimes_{k=1}^n(u_k u_k v_k + v_k u_k u_k))
=& \bar p_* \circ \gamma_3 ( \Sigma_{j=1}^{2^n} \iota_1\otimes_{k=1}^n x_k^j y_k^j z_k^j ) \\
=& \Sigma_{j=1}^{2^n} \bar p_* \circ \gamma_3 ( \iota_1\otimes_{k=1}^n x_k^j y_k^j z_k^j ) \\
=& 2^n \iota_1\otimes_{k=1}^n(\iota_{|u_k|+|v_k|}\otimes u_k)\\
=& 0.
\end{align*}
Therefore we have
\begin{align*}
&(\bar p\circ\theta)_*(\iota_1\otimes_{k=1}^n(u_k u_k v_k + v_k u_k u_k))\\
=&\bar p_* \circ \gamma_1 (\iota_1\otimes_{k=1}^n(u_k u_k v_k + v_k u_k u_k))
+\bar p_* \circ \gamma_2 (\iota_1\otimes_{k=1}^n(u_k u_k v_k + v_k u_k u_k))+\\
&\bar p_* \circ \gamma_3 (\iota_1\otimes_{k=1}^n(u_k u_k v_k + v_k u_k u_k))
+\bar p_* \circ \gamma_4 (\iota_1\otimes_{k=1}^n(u_k u_k v_k + v_k u_k u_k))\\
=&\iota_1\otimes_{k=1}^n(\iota_{|u_k|+|v_k|}\otimes u_k).
\end{align*}
Similarly we have
\begin{align*}
	(\bar p\circ\theta)_*(\iota_1\otimes_{k=1}^n(v_k v_k u_k + u_k v_k v_k))=\iota_1\otimes_{k=1}^n(\iota_{|u_k|+|v_k|}\otimes v_k).
\end{align*}
Thus
\begin{align*}
	\Sigma\wedge_{k=1}^{n} \Sigma^{|u_k|+|v_k|}{X_k} \stackrel{\bar i}{\rightarrow} \Sigma\wedge_{k=1}^{n} {X_k}^{(3)} \stackrel{\theta}{\rightarrow}
	\Sigma\wedge_{k=1}^{n} {X_k}^{(3)} \stackrel{\bar p}{\rightarrow} \Sigma\wedge_{k=1}^{n} \Sigma^{|u_k|+|v_k|}{X_k}
\end{align*}
induces an isomorphism on mod--$2$ homology.
Since the spaces considered are CW--complexes of finite type,
the composite $\bar p\circ \theta \circ \bar i$ is a homotopy equivalence.
Because $\theta$ factors through $\tilde L_{3}(\wedge_{k=1}^{n} {X_k})$, the statement follows.
\end{proof}
The following Proposition \ref{thm-wudecom} is due to Paul Selick and Jie Wu.
The case when $X=\Sigma X'$ is a suspension is shown in \cite{MR1650426};
if $X$ is not a suspension, one can modify an idea of Paul Selick and Jie Wu in \cite{MR2221079} to prove this proposition.
\begin{prop}[Paul Selick and Jie Wu]
\label{thm-wudecom}
	Let $X$ be a path--connected $2$--local CW--complex of finite type.
	Let $2<k_1<k_2<...$ be all the odd prime numbers in increasing order.
	Then there exists a topological space $A$ such that 
\begin{align*}
	\Omega\Sigma X \simeq \prod_j \Omega\tilde{(L_{k_j}}(X))\times A
\end{align*}
localized at $2$.
\begin{flushright}
$\Box$
\end{flushright}
\end{prop}
Theorem \ref{thm1} and \ref{thm2} are consequences of Proposition \ref{prop2}, \ref{prop3} and \ref{thm-wudecom}.
\begin{proof}[Proof of Theorem \ref{thm1}]
By Proposition \ref{prop3} and \ref{thm-wudecom},
the following map has a left homotopy inverse:
\begin{align*}
	\Omega\Sigma\wedge_{i=1}^{n} \Sigma^{|u_i|+|v_i|}{X_i} \rightarrow \Omega\Sigma\wedge_{i=1}^{n} {X_i}.
\end{align*}
For each $i$, the space $\Sigma^{|u_i|+|v_i|}{X_i}$ is again a $2$--cell complex.
By induction,
we can conclude that $\Omega\Sigma\wedge_{i=1}^{n} \Sigma^{\frac{3^k-1}{2}(|u_i|+|v_i|)}{X_i}$
retracts off
$\Omega\Sigma\wedge_{i=1}^{n} {X_i}$.
\end{proof}
\begin{proof}[Proof of Theorem \ref{thm2}]
Recall from in Proposition \ref{prop2} that the following map has a left homotopy inverse:
\begin{align*}
	\Sigma^{1+k(|u|+|v|)} X \rightarrow \tilde{L}_{2k+1}(X).
\end{align*}
Therefore $\Omega\Sigma^{1+k(|u|+|v|)}X$
retracts off
$\Omega\tilde{L}_{2k+1}(X)$.
The statement follows from Proposition \ref{thm-wudecom}.
\end{proof}
\section {Proof of Corollary \ref{coro1} and \ref{coro2} and Some Remarks}
First we give proofs of Corollary \ref{coro1} and \ref{coro2}.
\begin{proof}[Proof of Corollary \ref{coro1}]
Theorem \ref{thm1} implies that $\pi_{m}(\Sigma\wedge_{i=1}^{n} \Sigma^{\frac{3^k-1}{2}(|u_i|+|v_i|)}{X_i}) $ is a summand of $\pi_{m}(\Sigma\wedge_{i=1}^{n} {X_i})$ for $m\geq 1$.
For $k\geq 0$,
let $b_k=\Sigma_{i=1}^{n}(\frac{3^k-1}{2}(|u_i|+|v_i|))$.
When $k$ is large enough such that $j\leq b_k+2m$,
the Freudenthal suspension theorem implies
\begin{align*}
	\pi_{j+b_k}(\Sigma\wedge_{i=1}^{n} \Sigma^{\frac{3^k-1}{2}(|u_i|+|v_i|)}{X_i}) \cong \pi^s_{j+b_k}(\Sigma\wedge_{i=1}^{n} \Sigma^{\frac{3^k-1}{2}(|u_i|+|v_i|)}{X_i}).
\end{align*}
Thus
\begin{align*}
	\pi^s_{j+b_k}(\Sigma\wedge_{i=1}^{n} \Sigma^{\frac{3^k-1}{2}(|u_i|+|v_i|)}{X_i}) \cong \pi^s_j(\Sigma\wedge_{i=1}^{n} {X_i})
\end{align*}
The statement follows.
\end{proof}
\begin{proof}[Proof of Corollary \ref{coro2}]
As an immediate consequence of Theorem \ref{thm2},
for $k\geq 1$ and $n\geq 2$, the following map has a left homotopy inverse
\begin{align*}
	\Omega P^{n+1+k(2n-1)}(2)\hookrightarrow \Omega P^{n+1}(2).
\end{align*}
Further,
for $0\leq m \leq 3$,
the following map has a left homotopy inverse:
\begin{align*}
	\Omega P^{(4+8k)n+(2k+1)m-3k}(2)\hookrightarrow \Omega P^{4n+m}(2).
\end{align*}
Cohen and Wu \cite{MR1320988} showed that 
if $n\geq 4$ and $n\equiv 1 \pmod 2 $,
then $\pi_{4n-2}(P^{2n}(2))$ has a $\mathbb{Z}/8\mathbb{Z}$--summand.
Therefore,
if $(4+8k)n+(2k+1)m-3k\equiv 2 \pmod 4$,
then $\pi_*(\Omega P^{(4+8k)n+(2k+1)m-3k}(2))$ has a $\mathbb{Z}/8\mathbb{Z}$--summand,
The statement follows when we set $m=0,1,2,3$.
\end{proof}
Next we remark that Beben and Wu's result \cite[Proposition 5.2]{mrth} 
can be combined with Corollary \ref{coro1} to give a uniform formula.
First recall Beben and Wu's result.
\begin{prop}{\rm \cite[Proposition 5.2]{mrth}}
\label{prop-pb}
Let $X$ be the $p$--localization of a suspended CW--complex. Set $V=\bar H_*(X;\mathbb{Z}/p\mathbb{Z})$.
Let $M$ denote the sum of the degrees of the generators of $V$.
Define the sequence of integers $b_i$ recursively, with $b_0=0$ and
\begin{align*}
	b_i = (1+dim V) b_{i-1} + M.
\end{align*}
Let $V=\bar H_*(X;\mathbb{Z}/p\mathbb{Z})$, $1<dim V <p-1$ and $V_{odd}=0$ or $V_{even}=0$. Assume $X$ is $(m-1)$--connected for some $m\geq 1$.
Then for each $j$,
the stable homotopy group $\pi_j^s(\Sigma X)$ is a homotopy retract of $\pi_{j+b_i}(\Sigma X)$ for $i$ large enough such that $j\leq b_i+2m$.
\begin{flushright}
$\Box$
\end{flushright}
\end{prop}
Notice that when we set $X=\wedge_{i=1}^{n} {X_i}$, all the $b_i$ in Corollary \ref{coro1} are the same as the $b_i$ in Proposition \ref{prop-pb}.
Combine Proposition \ref{prop-pb} and Corollary \ref{coro1} to obtain the follows.
\begin{thm}
\label{prop-combprop}
Let $X$ be the $p$--localization of a suspended CW--complex.
Set $V=\bar H_*(X;\mathbb{Z}/p\mathbb{Z})$.
Let $M$ denote the sum of the degrees of the generators of $V$.
Define the sequence of integers $b_i$ recursively by setting $b_0=0$ and
\begin{align*}
	b_i = (1+dim V) b_{i-1} + M.
\end{align*}
Assume $X$ is $(m-1)$--connected for some $m\geq 1$
and let $V=\bar H_*(X)$, 
If either one of the following is satisfied:
\begin{itemize}
	\item $1<dim V <p-1$, and $V_{odd}=0$ or $V_{even}=0$.
	\item $2=p=dim(W_i)$, and $X=\wedge_{i=1}^{n} {X_i}$ with $W_i=\bar H_*(X_i)$ for $1\leq i \leq n$.
\end{itemize}
Then for each $j$,
the stable homotopy group $\pi_j^s(\Sigma X)$ is a homotopy retract of $\pi_{j+b_i}(\Sigma X)$ for $i$ large enough such that $j\leq b_i+2m$.
\begin{flushright}
$\Box$
\end{flushright}
\end{thm}

%
%
%
\bibliographystyle{gtart}
\bibliography{cdata}
%



\end{document}